\documentclass[12pt]{article}
\usepackage{a4}
\usepackage{amsmath}
\usepackage{amssymb}
\usepackage{amsfonts}
\usepackage{amsthm}
\usepackage{graphicx}
 \newtheorem{thm}{Theorem}[section]
 \newtheorem{cor}[thm]{Corollary}
 \newtheorem{prop}[thm]{Proposition}

 \newtheorem{rem}[thm]{Remark}
 
 \DeclareMathOperator{\supp}{supp}

\DeclareMathOperator{\Arccos}{Arccos}

\newcommand{\Int}{\int_{-\infty}^\infty}
\newcommand{\om}{\omega}
\renewcommand{\a}{\alpha}

\renewcommand{\d}{\delta}
\renewcommand{\span}{span}
 
\def\C{\mathbb{C}}
\def\R{\mathbb{R}}
\def\S{\mathbb{S}}

\def\Z{\mathbb{Z}}

\author{Christian Berg}
\title{A  unified view of space-time covariance functions through Gelfand pairs}

\date{1.10.2020}

\begin{document}
\maketitle

\begin{abstract} We give a characterization of positive definite integrable functions on a product of two Gelfand pairs as an integral of positive definite functions on one of the Gelfand pairs with respect to the Plancherel measure on the dual of the other Gelfand pair. 

In the very special case where the Gelfand pairs are Euclidean groups and the compact subgroups are reduced to the identity, the characterization is a much cited result  in spatio-temporal statistics due to Cressie, Huang and Gneiting.

 When one of the Gelfand pairs is compact the characterization leads to  results about
expansions in spherical functions with positive definite expansion functions, thereby recovering recent results of the author in collaboration with Peron and Porcu. In the special case when the compact Gelfand pair consists of orthogonal groups, the characterization is important in geostatistics and covers a recent result of Porcu and White.   
\end{abstract}



\section{Introduction}
In connection with analysis of spatiotemporal datasets it has been valuable to have knowledge about covariance functions on $\R^d\times\R$ or $\S^d\times\R$, where $\S^d$ denotes the unit sphere in $\R^{d+1}$, and $\R$ is a model for time. During the last 30 years there has been numerous papers in this area as witnessed e.g. by the recent survey paper by Porcu, Furrer and Nychka \cite{P:F:N} with more than 200 citations.

Covariance functions are positive definite and conversely, by a Theorem of
Kolmogorov, every positive definite function is the covariance function of some Gaussian random field. For an introduction  to these concepts see e.g. \cite{M:P}.

In \cite{C:H} Cressie and Huang considered continuous positive definite functions $C:\R^d\times\R\to\R$ and proposed a way to produce such functions, see Proposition~\ref{thm:CH}. 

Berg and Porcu \cite{B:P} extended the seminal work of  Schoenberg  
\cite{Sc} to expansions of positive definite functions on $\S^d\times \R$, but realized that the same  type of result holds, if the temporal space $\R$ is replaced by an arbitrary locally compact group $L$. Later Berg, Peron and Porcu \cite{B:P:P} showed how to extend the framework further from products $\S^d\times L$ to products of compact Gelfand pairs and locally compact groups $L$. The sphere $\S^d$ appears as the homogeneous space $O(d+1)/O(d)$, where $O(d)$ is the orthogonal group  in $\R^d$. The pair of groups $(O(d+1), O(d))$ is an example of a compact Gelfand pair. In a recent paper \cite{P:C:G:W:A} the authors consider covariance functions on products $\S^{d_1}\times \S^{d_2}$ of spheres. All these results will be covered by our treatment of harmonic analysis on products of Gelfand pairs.

Let us motivate the results to follow by outlining the main idea of \cite{C:H} and following their terminology.

Consider a correlation function $C:\R^d\times\R\to\R$ given as
\begin{equation}\label{eq:CH1}
C(\mathbf{h};u)=\int\int e^{i\boldsymbol{h'\omega}+iu\tau}g(\boldsymbol{\omega};\tau)d\boldsymbol{\omega}\,d\tau,
\end{equation}
where $g$ is a continuous probability density on $\R^d\times\R$. In \eqref{eq:CH1} and below $\boldsymbol{h,\omega}$ are column vectors in $\R^d$ and $\boldsymbol{h'\omega}$ is their scalar product.

Under the assumption that $C$ is integrable, the authors   of \cite{C:H} write
\begin{equation}\label{eq:CH2}
g(\boldsymbol{\omega};\tau)=(2\pi)^{-d-1}\int\int  e^{-i\boldsymbol{h'\omega}-iu\tau}C(\mathbf{h};u)d\mathbf{h}\,du=(2\pi)^{-1}\int e^{-iu\tau}h(\boldsymbol{\omega};u)\,du,
\end{equation}
where
\begin{equation}\label{eq:CH3}
h(\boldsymbol{\omega};u)=(2\pi)^{-d}\int  e^{-i\boldsymbol{h'\omega}}C(\mathbf{h};u)d\mathbf{h}=\int e^{iu\tau}g(\boldsymbol{\omega};\tau)\,d\tau.
\end{equation}
Here the first equality in \eqref{eq:CH2} follows from the Inversion Theorem for Fourier integrals on $\R^{d+1}$, and the second equality is a consequence of Fubini's Theorem, but has to be used with some care, because the function $h$ is only defined for almost all $u\in\R$, since the function
$C(\cdot;u)$ is only Lebesgue integrable on $\R^d$ for $u\in\R$ outside a Lebesgue null set. 

The second equality in \eqref{eq:CH3} comes from the Inversion Theorem for Fourier integrals on $\R$ again without making the assumptions precise for the Inversion Theorem to hold.
 
Cressie and Huang then assume that $h$ can be written
\begin{equation}\label{eq:CH4}
h(\boldsymbol{\omega};u)=\rho(\boldsymbol{\omega};u)k(\boldsymbol{\omega}) 
\end{equation}
with the following assumptions

(C1) For each $\boldsymbol{\omega}\in\R^d, \rho(\boldsymbol{\omega};\cdot)$ is a continuous correlation function such that  $\int|\rho(\boldsymbol{\omega};u)|\,du<\infty$  and $k(\boldsymbol{\omega})>0$.

(C2) $\int k(\boldsymbol{\omega})\,d\boldsymbol{\omega}<\infty$.

Note that $k(\boldsymbol{\omega})=h(\boldsymbol{\omega};0)$.
The result of Cressie and Huang can be stated as follows, where we have reformulated the two conditions without the factor $k(\boldsymbol{\omega})$ and using the terminology of Section 2:
 
\begin{prop}\label{thm:CH} Assume that $h:\R^d\times \R\to\C$ is a continuous function satisfying 

(C1') For each $\boldsymbol{\omega}\in\R^d, h(\boldsymbol{\omega};\cdot)\in\mathcal P(\R)\cap L^1(\R)$.  

(C2') $\int_{\R^d} h(\boldsymbol{\omega};0)\,d\boldsymbol{\omega}<\infty$.

 Then $C:\R^d\times \R\to\C$ defined by
\begin{equation}\label{eq:CH5}
C(\mathbf{h};u):=\int_{\R^d} e^{i\boldsymbol{h'\omega}} h(\boldsymbol{\omega};u)\,d\boldsymbol{\omega}
\end{equation}
belongs to $\mathcal P(\R^d\times\R)$.
\end{prop}

In Proposition~\ref{thm:CHgp} we give a far reaching generalization of the above Proposition. We stress that the functions $C\in\mathcal P(\R^d\times\R)$ given by \eqref{eq:CH5} are not necessarily integrable, see the Appendix, even though integrability 
of $C$ was a motivating assumption.

This result  by Cressie and Huang has been taken up and refined by several authors, see
\cite{Gn}, \cite{P:G:M}, \cite{A:G}.

The result of \cite[p. 598]{Gn} can be stated like this, where  we use the mathematical terminology "continuous positive definite function" instead of the statistical terminology "covariance function".

\begin{thm}\label{thm:Gn} Let $k, l$ be positive integers. A continuous, bounded, symmetric and integrable function $C:\R^k\times\R^l\to \R$ is positive definite if and only if 
\begin{equation}\label{eq:A2}
C_{\boldsymbol{\omega}}(\mathbf{u})=\int e^{-i\boldsymbol{h'\omega}} C(\mathbf{h};\mathbf{u}) d\mathbf{h},\quad \mathbf{u}\in\R^l
\end{equation} 
is a continuous positive definite function for almost all $\boldsymbol{\omega}\in\R^k$.
\end{thm}

From a mathematical point of view Equation \eqref{eq:A2} has to be made precise in the sense that $C(\cdot;\mathbf{u})$ need only be integrable on $\R^k$ for almost all $\mathbf{u}\in\R^l$ by Fubini's Theorem, so for $\mathbf{u}$ in a null-set the Fourier transform  in \eqref{eq:A2}
is not well defined.

This paper gives a unifying framework that encompasses all the relevant contributions in the construction and characterization of space-time covariance functions. Specifically we have chosen to present our results using the abstract language of locally compact groups (\cite{D}, \cite{S}) and harmonic analysis on Gelfand pairs (\cite{Dieu}, \cite{vD}, \cite{W}).
The latter includes the classical framework for harmonic analysis on locally compact abelian groups (\cite{R}, \cite{B:F}, \cite{S}) and the case of homogeneous spaces like the  spheres $\S^d$ (\cite[Chap. 9]{F}, \cite{Sc}, \cite{D:X}). We have tried to formulate the results under minimal assumptions  and to introduce stringent mathematical formulation of previous results that are mathematically inaccurate.

In this way  the work of several authors like \cite{A:G}, \cite{C:H}, \cite{Gn}, \cite{P:C:G:W:A}, \cite{P:G:M}, \cite{P:W} will be special cases of our main Theorem~\ref{thm:CHG*}.  

We give some background material about harmonic analysis on Gelfand pairs  in Section 2.
The main results in harmonic analysis on Gelfand pairs are analogues of Bochner's Theorem, the Inversion Theorem and Plancherel's Theorem, which are classical theorems
about locally compact abelian groups. We also need an addition to Bochner's Theorem. It is presented in Theorem ~\ref{thm:God-int}, and we call it the Integrable Bochner Theorem. We cannot claim anything new in this theorem, but we have only been able to find it formulated in an old paper of ours, see \cite[Corollary 2.3]{B}, or under unnecessary  extra conditions, see \cite{S}, Theorems 1.9.8 and 1.9.12. 

 In Section 3 we formulate our main result Theorem~\ref{thm:CHG*}, which we call the Cressie-Huang-Gneiting Theorem about products of two Gelfand pairs. The name is chosen to acknowledge the inspiring results of \cite{C:H} and \cite{Gn}.
Theorem~\ref{thm:Gn} is a special case of this result when $\R^k$ and $\R^l$ are considered as abelian Gelfand pairs.

 Corollary~\ref{thm:cor1} treats the special case, where the first Gelfand pair is compact, leading to series expansions in spherical functions.

 In Section 4 we discuss Corollary~\ref{thm:cor1} in the special case, where the first Gelfand pair is $(O(d+1),O(d))$ and the second is abelian. This covers the Theorem of Porcu and White in \cite{P:W} about covariance functions on spheres cross time.

In Section 5 we deviate from the main theme by revisiting a result of Gneiting. In \cite{Gn1} he showed that a completely monotonic function  defines positive definite functions on spheres $\S^d$ of any dimension. We give a new proof of this result. The new proof has the advantage over the original proof of Gneiting
by giving expressions  for the power series coefficients in terms of a family of polynomials related to exponential Bell partition polynomials.

In the Appendix we prove the existence of  a strictly positive, bounded, continuous and integrable function $f$ on $\R$ for which the Fourier transform $\mathcal F f$ is not integrable. This shows that the construction of Cressie and Huang in \cite{C:H} does not lead to all integrable positive definite functions on the product $\R^d\times\R$.

\section{Harmonic analysis on Gelfand Pairs}
For a locally compact topological space $X$ endowed with a positive Radon measure $\mu$ we denote by $L^p(X,\mu)=L^p(\mu)$ the Banach space of equivalence classes of measurable functions  whose $p$'th power ($1\le p<\infty$) is $\mu$-integrable. 

In the following $G$ denotes a locally compact group with neutral element $e_G$ and left Haar measure $\om_G$. A  function $f:G\to\mathbb C$ is called positive definite if
for any $n\in\mathbb N$ and any $u_1,\ldots,u_n\in G$ the $n\times n$-matrix 
$$
[f(u_k^{-1}u_l)]_{k,l=1}^n
$$
is positive semidefinite, {\em i.e.}, for any $(a_1,\ldots,a_n)\in\C^n$
\begin{equation}\label{eq:pdpointnw}
\sum_{k,l=1}^n f(u_k^{-1}u_l)a_k\overline{a_l}\ge 0, 
\end{equation}
see e.g. \cite[p. 255]{D}, where $f$ is said to be of positive type, or \cite[p.14]{S}. It is known that a positive definite function $f$ is bounded and satisfies 
\begin{equation}\label{eq:pdele}
f(u^{-1})=\overline{f(u)},\quad |f(u)|\le f(e_G),\quad u\in G.
\end{equation}

 The set of continuous and positive definite functions on $G$ is denoted
$\mathcal P(G)$. 
 
 If the group $G$ is abelian, we use the additive notation, and the neutral element is denoted $0$. We shall mainly reserve the letter $A$ for locally compact abelian groups,  LCA-groups in short. Concerning harmonic analysis on LCA-groups $A$, we refer to \cite{R}. The dual group $\widehat{A}$ consists of all continuous group homomorphisms $\gamma:A\to\mathbb T$, called characters. Here $\mathbb T$ is the unit circle in the complex plane, considered as a multiplicative group. The group operation of $\widehat{A}$ is pointwise multiplication. 

 The Fourier transform of a function $f\in L^1(\om_A)=L^1(A,\om_A)$ is denoted $\mathcal Ff$ or $\mathcal F_Af$, if it is necessary to mention the group, and is defined by
 \begin{equation}\label{eq:FT}
 \mathcal F_Af(\gamma)=\mathcal Ff(\gamma)=\int_A \overline{\gamma(u)}f(u)\,d\om_A(u),\quad \gamma\in\widehat{A}.
 \end{equation}

Harmonic analysis on LCA-groups is a special case of harmonic analysis on Gelfand pairs $(G,K)$, where $G$ is a locally compact group and $K$ is a compact subgroup of $G$ such that the convolution algebra of $K$-bi-invariant continuous functions with compact support on $G$ is commutative.  We use the terminology of \cite{B:P:P}, which contains a short introduction to Gelfand pairs, but refer to \cite{Dieu}, \cite{vD} and \cite{W} for detailed information about Gelfand pairs. In \cite{Dieu} all groups are assumed metrizable and separable, but this is no restriction for the applications to statistics. For any Gelfand pair $(G,K)$ the group  $G$ is unimodular, cf. \cite{B}, so the left invariant Haar measure $\om_G$ is also right invariant. For functions $f_1, f_2: G\to\C$ their convolution $f_1\star f_2$ is then defined as
$$
f_1\star f_2(x)=\int_G f_1(y)f_2(y^{-1}x)d\om_G(y)=\int_G f_1(xy^{-1})f_2(y)d\om_G(y),
\quad x\in G,
$$
whenever the integrals make sense.

 For a set $\mathcal E$ of functions on $G$ we denote by $\mathcal E_K^\sharp$ the set of functions from $\mathcal E$ which are bi-invariant under $K$.

For $f\in C(G)$, {\em i.e.}, $f:G\to\C$ is continuous, we define
\begin{equation}\label{eq:proj}
f^\sharp(u):=\int_K\int_K f(kuk')d\om_K(k) d\om_K(k'),\quad u\in G,
\end{equation}
where $\om_K$ is Haar measure on $K$ normalized to $\om_K(K)=1$.
Then $f^\sharp\in C^\sharp_K(G)$ and 
$$
\int_G f(u)d\om_G(u)=\int_G f^\sharp(u)d\om_G(u)
$$
provided $f\in L^1(\om_G)$.

The compact subgroup $K$ determines an equivalence relation $\sim$ in $G$ defined by $x\sim y$ if and only if $x=kyk'$ for some $k,k'\in K$. The equivalence classes are the compact sets $KxK,x\in G$, which are called double cosets. The set of double cosets is denoted $K\backslash  
G/K$, and it is a locally compact space in the quotient topology, which by definition is the finest topology making the mapping $x\mapsto KxK$ from $G$ to $K\backslash G/K$ continuous. Functions on $K\backslash G/K$  can be identified with functions on $G$ which are bi-invariant with respect to $K$.

For an LCA-group $A$ the pair $(A,\{0\})$ is a Gelfand pair called abelian.
The characters for LCA-groups are replaced by positive definite spherical functions for  Gelfand pairs $(G,K)$. A spherical function is a continuous function $\varphi:G\to \C$ satisfying $\varphi(e_G)=1$ and
\begin{equation}\label{eq:sph}
\int_K \varphi(ukv) d\om_K(k)=\varphi(u)\varphi(v),\quad u,v\in G.
\end{equation}
It is automatically bi-invariant under $K$. For abelian Gelfand pairs Equation \eqref{eq:sph} becomes the homomorphism property of characters.

 For $f\in L^1(\om_G)_K^\sharp$ the Fourier transform is the function $\mathcal F f:Z\to\C$ defined by
\begin{equation}\label{eq:Gf}
\mathcal F f(\varphi)=\int_G \overline{\varphi(u)}f(u)\,d\om_G(u),\quad \varphi\in Z,
\end{equation}
where $Z$ is the set of positive definite spherical functions for $(G,K)$, called the dual space of the Gelfand pair. It carries a locally compact topology such that $\varphi_j\to\varphi$ in this topology if and only if $\varphi_j(u)\to\varphi(u)$ uniformly for $u$ in compact subsets of $G$.   Recall from \eqref{eq:pdele} that $\varphi\in Z$ has the property $|\varphi(u)|\le \varphi(e_G)=1$ for $u\in G$. If several Gelfand pairs $(G,K)$ are present, we write $\mathcal F_G f$ instead of $\mathcal F f$. 

We construct an approximative identity $(\rho_V^\sharp)_{V\in\mathcal V}$ for a Gelfand pair  in the following way: Let $C_c(G)$ denote the set of continuous functions on $G$ with compact support, and let $\mathcal V$ denote the downwards filtering family of compact neighbourhoods $V$ of $e_G$ in $G$. For $V\in\mathcal V$ we choose $\rho_V\in C_c(G)_+$ satisfying $\supp(\rho_V)\subseteq V$ and $\int\rho_V d\om_G=1$. Therefore $\rho_V^\sharp\in C_c(G)_K^\sharp$ is non-negative with integral 1. Given a bi-invariant function or measure $f$, the general principle is that $f\star\rho_V^\sharp\to f$ as $V$ shrinks to $e_G$ in a topology appropriate for $f$. For $f\in C(G)_K^\sharp$ the convergence is uniform on compact subsets of $G$. Notice that $\rho_V^\sharp d\om_G$ converges vaguely to $\om_K$ when $V\to e_G$.

There exists a positive Radon measure $\nu$ on $Z$ called the Plancherel measure.
It is the unique measure such that
\begin{equation}\label{eq:Plan}
\int_G |f(u)|^2\,d\om_G(u)=\int_Z |\mathcal F f(\varphi)|^2 d\nu(\varphi)
\end{equation}
for all $f\in C_c(G)_K^\sharp$.

The mapping $\varphi \mapsto\overline{\varphi}$ is an involutive homeomorphism of $Z$, and $\nu$ is invariant under this mapping:
\begin{equation}\label{eq:nuinv}
\int_Z h(\overline{\varphi})\,d\nu(\varphi)=\int_Z h(\varphi)\,d\nu(\varphi),\quad h\in L^1(\nu).
\end{equation}  

We need the Inversion Theorem for Gelfand pairs in the following form generalizing
 \cite[Corollary 1.21]{S:W} for Euclidean groups. It is a special case of the Inversion Theorem in \cite[p. 140]{B}:

\begin{thm}\label{thm:InvTh} Let $(G,K)$ be a Gelfand pair and let $f\in L^1(\om_G)_K^\sharp$ be such that  $\mathcal F f\in L^1(\nu)$. Then
$$
f(u)=\int_Z \varphi(u) \mathcal F f(\varphi) \,d\nu(\varphi),\quad \om_G\;\mbox{a.e. in $G$}. 
$$
\end{thm}

We recall the Bochner-Godement Theorem, cf. \cite[Theorem 2.1]{B:P:P}, \cite[Theorem 6.4.4]{vD}, according to which $f\in\mathcal P^\sharp_K(G)$ if and only if there exists a (uniquely determined)  positive finite Radon measure $\mu$ on $Z$ such that
\begin{equation}\label{eq:B-G}
f(u)=\int_Z \varphi(u)\,d\mu(\varphi),\quad u\in G.
\end{equation}

The following result is useful for verifying that a given function $f$ is positive definite, when the measure $\mu$ is unknown. It is a special case of Corollary 2.3 in \cite{B} and is essential for the following.

\begin{thm}[The Integrable Bochner Theorem]\label{thm:God-int} Let $(G,K)$ be a Gelfand pair and let $f:G\to\C$ be a continuous and integrable function which is bi-invariant under $K$. Then $f$ is positive definite if and only if  
$$
\mathcal F f(\varphi)\ge 0,\quad \varphi\in \supp(\nu).
$$

If the equivalent conditions hold, then $\mathcal F f\in L^1(\nu)$ and
$$
f(u)=\int_Z \varphi(u)\mathcal F f(\varphi)\,d\nu(\varphi),\quad u\in G.
$$
\end{thm}

\begin{proof} For the convenience of the reader we give a proof.

(i) Assume first that $\mathcal F f(\varphi)\ge 0$ for $\varphi\in\supp(\nu)$.

Under the additional hypothesis that $\mathcal F f\in L^1(\nu)$, we get by  
 Theorem~\ref{thm:InvTh}
$$
f(u)=\int_Z \varphi(u)\mathcal F f(\varphi) d\nu(\varphi),\quad u\in G,
$$
and we have equality for all $u\in G$, because $f$ is assumed continuous.
For $h\in C_c(G)$ we then have ($\tilde{h}(u):=\overline{h(u^{-1})},\;u\in G$),   
\begin{eqnarray*}
&&\int_G f(u) \tilde{h}\star h(u)\,d\om_G(u)=
\int_G \tilde{h}\star h(u)\int_Z \varphi(u)\mathcal F f(\varphi)d\nu(\varphi) d\om_G(u)\\
&=&
\int_Z \mathcal F f(\varphi) \int_G \varphi(u) \tilde{h}\star h(u) d\om_G(u) d\nu(\varphi)
 \ge 0,
\end{eqnarray*}
because
$$
\int_G \varphi(u) \tilde{h}\star h(u) d\om_G(u)\ge 0
$$
for all $\varphi\in Z$, and then $f\in\mathcal P(G)$, cf. \cite[p. 256]{D}.

Without the assumption $\mathcal F f\in L^1(\nu)$, we consider $f_V:=f \star \rho_V^\sharp \star (\rho_V^\sharp\tilde{)}$, using an approximative identity.
Since $\mathcal F f_V=\mathcal F f |\mathcal F \rho_V^\sharp|^2$ is non-negative and $\nu$-integrable by  \eqref{eq:Plan}, we get $f_V\in\mathcal P(G)$ by what has just been proved. We next use that $f_V(u)\to f(u)$ locally uniformly for $u\in G$ as $V$ shrinks to $e_G$, and we get $f\in\mathcal P(G)$.

(ii) Conversely, if $f$ is assumed positive definite, we know from Bochner-Godement's Theorem that \eqref{eq:B-G} holds
for a uniquely determined positive finite measure $\mu$ on $Z$. However, by the Inversion Theorem in the form of \cite[p. 84-85]{vD}, the integrability of $f$ implies that $\mathcal F f\in L^1(\nu)$ and $d\mu=\mathcal F f d\nu$, and this forces $\mathcal F f$ to be non-negative on $\supp(\nu)$.
\end{proof}

\begin{rem}\label{thm:reducedsupp}
{\rm The reader is warned that the support of the Plancherel measure $\supp(\nu)$ can be strictly smaller than $Z$. This cannot happen for compact Gelfand pairs $(G,K)$ ({\em i.e.},  if $G$ is compact) and for abelian Gelfand pairs $(A, \{0\})$.

  For compact Gelfand pairs the dual space $Z$ is discrete and for each $\varphi\in Z$ we consider a certain vector space
\begin{equation}\label{eq:hfi}
H_\varphi:=\span \{\varphi_g \mid g\in G\}.
\end{equation} 
 Here $\varphi_g(x):=\varphi(g^{-1}x)$ is a right invariant continuous function on $G$ considered as a continuous function on the homogeneous space $G/K$. It is a classical fact that $H_\varphi$
 is of finite dimension, which we denote by $\d(\varphi)$, and the Plancherel measure is $\nu(\{\varphi\})=\d(\varphi)$.
cf. \cite[p.86]{vD} or \cite[p. 265]{B:P:P}.

  In the abelian case $Z$ is the dual group $\widehat{A}$, and $\nu$ is the Haar measure on $\widehat{A}$ called dual to $\om_A$, meaning that the Inversion Theorem holds, cf. \cite[p. 24]{R}. 
}
\end{rem}

For the reader's convenience we state Theorem~\ref{thm:God-int} for the important special
cases of abelian and compact Gelfand pairs.

\begin{cor}\label{thm:abgp} Let $A$ be an LCA-group and $f:A\to\C$ a continuous and integrable function. 

Then $f\in\mathcal P(A)$ if and only if 
$$
\mathcal F f(\gamma)\ge 0, \quad\gamma\in\widehat{A}.
$$
If the equivalent conditions hold, then $\mathcal F f\in L^1(\om_{\widehat{A}})$ and
$$
f(u)=\int_{\widehat{A}} \gamma(u)\mathcal F f(\gamma) d\om_{\widehat{A}}(\gamma),\quad u\in A.
$$
\end{cor}

\begin{cor}\label{thm:comgp} Let $(G,K)$ be a compact Gelfand pair and $f\in C(G)_K^\sharp$.

Then $f\in\mathcal P(G)$ if and only if 
$$
\mathcal F f(\varphi)\ge 0,\quad \varphi\in Z.
$$
If the equivalent conditions hold, then $\mathcal F f\in L^1(\nu)$ and
$$
f(u)=\sum_{\varphi\in Z} \delta(\varphi) \varphi(u)\mathcal F f(\varphi), \quad u\in G.
$$
The expansion is absolutely and uniformly convergent for $u\in G$.
\end{cor}

\begin{rem} {\rm If $f:A\to\C$ is integrable without being continuous on an LCA-group $A$, then it is possible that $\mathcal F f\ge 0$ without $f$ being positive definite in the classical sense and in particular bounded. We give an example with $A=\R$.

 Linnik's probability densities $p_\a, 0<\a\le 1$ are symmetric on $\R$ with
$$
\mathcal Fp_\a(t)=(1+|t|^\a)^{-1},\quad t\in\R.
$$ 
They are given as 
$$
p_\a(x)=\frac{\sin(\a\pi/2)}{\pi}\int_0^\infty\frac{e^{-xy}y^\a}{|1+y^\a e^{i\a\pi/2}|^2} dy, \quad x\ge 0,
$$
see \cite{L}, \cite{K:O:H}. In particular 
$$
p_\a(0)=\lim_{x\to 0}p_\a(x)=\infty,
$$
and $p_\a$ is decreasing and continuous on $]0,\infty[$.

\medskip
Cooper \cite{Co} initiated a study of the set $\mathcal P(J)$ of measurable functions $f:\R\to\C$ satisfying
$$
\Int\Int f(x-y) h(x)\overline{h(y)}dx\,dy\ge 0,
$$
for all $h\in J$, where $J$ is a family of measurable functions $h:\R\to\C$.

Note that if $f$ is continuous then 
$$
f\in\mathcal P(C_c(\R)) \iff f\in\mathcal P(\R).
$$
Theorem 2 in \cite{P:S} states that for $f\in L^1(\R)$
 one has 
$$
\mathcal F f\ge 0 \iff f\in\mathcal P(L^2(\om_{\R})).
$$

}
\end{rem}

\section{The Cressie-Huang-Gneiting Theorem for Gel\-fand pairs}

\subsection{Products of two Gelfand pairs}

For two Gelfand pairs $(G_1,K_1), (G_2,K_2)$ with dual spaces $Z_1, Z_2$ the product $(G_1\times G_2, K_1\times K_2)$ is again a Gelfand pair. Its dual space can be identified with $Z_1\times Z_2$ in the following way: For $\varphi\in Z_1, \psi\in Z_2$
$$
\varphi\otimes\psi(x,y):=\varphi(x)\psi(y),\quad (x,y)\in G_1\times G_2
$$
is a positive definite spherical function for $(G_1\times G_2, K_1\times K_2)$, and all positive definite spherical functions for the product Gelfand pair are given in this way.

The Plancherel measure for $(G_1\times G_2, K_1\times K_2)$ is the product measure $\nu_1\otimes\nu_2$ of the Plancherel measures $\nu_i$ for $(G_i,K_i)$, $i=1,2$.

For $f\in L^1(G_1\times G_2, \om_{G_1\times G_2})_{K_1\times K_2}^\sharp$ let
\begin{equation}\label{eq:Hnull*}
G_1':=\{x\in G_1 \mid f(x,\cdot)\in L^1(G_2,\om_{G_2})\}.
\end{equation}
By Fubini's Theorem $\om_{G_1}(G_1\setminus G_1')=0$ and $G_1'$ is $K_1$ bi-invariant.
For all $x\in G_1'$ we have $f(x,\cdot)\in L^1(\om_{G_2})_{K_2}^\sharp$, and for those $x$ we define
\begin{equation}\label{eq:FLg*} 
\mathcal F_{G_2} f(x,\cdot)(\psi)=\int_{G_2}\overline{\psi(y)} f(x,y) d\om_{G_2}(y),\quad \psi\in Z_2.
\end{equation}

Similarly
\begin{equation}\label{eq:Hnull*2}
G_2':=\{y\in G_2 \mid f(\cdot,y)\in L^1(G_1,\om_{G_1})\}
\end{equation}
satisfies $\om_{G_2}(G_2\setminus G_2')=0$ and $G_2'$ is $K_2$ bi-invariant. For all $y\in G_2'$
we have $f(\cdot,y)\in L^1(\om_{G_1})_{K_1}^\sharp$, and for those $y$ we define
\begin{equation}\label{eq:FLh}
  \mathcal F_{G_1}f(\cdot,y)(\varphi)=\int_{G_1} \overline{\varphi(x)} f(x,y) d\om_{G_1}(x),\quad \varphi\in Z_1.
 \end{equation}

\begin{thm}\label{thm:CHG*} Let $(G_1,K_1)$ and $(G_2, K_2)$ be Gelfand pairs with dual spaces $Z_1$ and $Z_2$, and let $f:G_1\times G_2\to \C$ be a continuous and integrable function, bi-invariant under $K_1\times K_2$. 

The following conditions are equivalent:
 
(i) $f\in\mathcal P(G_1\times G_2)$.

(ii) For almost  all $\psi\in\supp(\nu_2)$   
the function $x\mapsto \mathcal F_{G_2} f(x,\cdot)(\psi)$ (defined for $x\in G_1'$) 
is equal almost everywhere in $G_1$ to a $K_1$ bi-invariant continuous positive definite  function denoted $p_{G_1}[f,\psi]$ on $G_1$. 

(iii) For almost all $\varphi\in\supp(\nu_1)$ the function  $y\mapsto \mathcal F_{G_1}f(\cdot,y)(\varphi)$ (defined for $y\in G_2'$) is equal almost everywhere in $G_2$ to a $K_2$ bi-invariant continuous positive definite function denoted $p_{G_2}[f,\varphi]$ on $G_2$.  

If the equivalent conditions are satisfied, then $p_{G_1}[f,\psi]\in L^1(\om_{G_1})$,
$p_{G_2}[f,\varphi]\in L^1(\om_{G_2})$ and for $(x,y)\in G_1\times G_2$
\begin{equation}\label{eq:fxy}
f(x,y)=\int_{\supp(\nu_1)}\varphi(x)p_{G_2}[f,\varphi](y)\,d\nu_1(\varphi)=
\int_{\supp(\nu_2)}\psi(y)p_{G_1}[f,\psi](x)\,d\nu_2(\psi),
\end{equation}
where the first integrand is defined for $\nu_1$-almost all $\varphi\in\supp(\nu_1)$ and the second integrand is defined for $\nu_2$-almost all $\psi\in\supp(\nu_2)$.
\end{thm}

\begin{proof} Since the pairs $(G_1,K_1)$ and $(G_2,K_2)$ appear symmetrically, it is enough to prove that (i) is equivalent to (ii).

Suppose first that $f\in\mathcal P(G_1\times G_2)$. By the Integrable Bochner Theorem~\ref{thm:God-int} we know that $\mathcal F_{G_1\times G_2}f(\varphi,\psi)\ge 0$ for
$\varphi\in \supp(\nu_1),\psi\in \supp(\nu_2)$ and $\mathcal F_{G_1\times G_2} f\in L^1(\nu_1\otimes\nu_2)$. Furthermore,
\begin{equation}\label{eq:22}
f(x,y)=\int_{Z_1\times Z_2} \varphi(x)\psi(y)\mathcal F_{G_1\times G_2} f(\varphi,\psi)d\nu_1\otimes\nu_2(\varphi,\psi), \quad (x,y)\in G_1\times G_2.
\end{equation}
We also have
\begin{eqnarray}\label{eq:fub1}
\mathcal F_{G_1\times G_2} f(\varphi,\psi)&=&\int_{G_1\times G_2}\overline{\varphi(x)}\,\overline{\psi(y)} f(x,y) d\om_{G_1}\otimes\om_{G_2}(x,y)\nonumber\\
&=&\int_{G_1'}\left(\overline{\varphi(x)} \int_{G_2} \overline{\psi(y)} f(x,y)\,d\om_{G_2}(y)\right) d\om_{G_1}(x)\nonumber\\
&=&\int_{G_1'} \overline{\varphi(x)} \mathcal F_{G_2}  f(x,\cdot)(\psi) d\om_{G_1}(x),
\end{eqnarray}
where $G_1'$ is as in \eqref{eq:Hnull*}.

Since $\mathcal F_{G_1\times G_2} f\in L^1(\nu_1\otimes\nu_2)$, Fubini's Theorem shows that
\begin{equation}\label{eq:nu1'}
\supp(\nu_1)':=\{\varphi\in\supp(\nu_1) \mid \mathcal F_{G_1\times G_2}f(\varphi,\cdot)\in L^1(\nu_2)\}
\end{equation}
satisfies $\nu_1(\supp(\nu_1)\setminus\supp(\nu_1)')=0$, and similarly
\begin{equation}\label{eq:nu2'}
\supp(\nu_2)':=\{\psi\in\supp(\nu_2) \mid \mathcal F_{G_1\times G_2}f(\cdot,\psi)\in L^1(\nu_1)\}
\end{equation}
satisfies $\nu_2(\supp(\nu_2)\setminus\supp(\nu_2)')=0$.

In particular \eqref{eq:fub1} holds for $\varphi\in\supp(\nu_1), \psi\in\supp(\nu_2)'$, so
  the Inversion Theorem~\ref{thm:InvTh} applied for $\psi\in\supp(\nu_2)'$ gives
\begin{equation}\label{eq:IT}
\mathcal F_{G_2} f(x,\cdot)(\psi)=\int_{Z_1}\varphi(x)\mathcal F_{G_1\times G_2} f(\varphi,\psi) \, d\nu_1(\varphi)
\end{equation}
for almost all $x\in G_1$.  
The right-hand side of this equation as a function of $x$ is a continuous positive definite function  on $G_1$ denoted $p_{G_1}[f,\psi]$, {\em i.e.},
\begin{equation}\label{eq:IT1}
p_{G_1}[f,\psi](x):=\int_{Z_1}\varphi(x)\mathcal F_{G_1\times G_2} f(\varphi,\psi) \, d\nu_1(\varphi),\quad x\in G_1.
\end{equation}

Clearly \eqref{eq:IT1}  is  bi-invariant under $K_1$. Furthermore, $p_{G_1}[f,\psi]\in L^1(\om_{G_1})$ because
\begin{eqnarray*}
&&\int_{G_1}|p_{G_1}[f,\psi](x)|\,d\om_{G_1}(x)=\\
&&\int_{G_1'} |\mathcal F_{G_2} f(x,\cdot)(\psi)| d\om_{G_1}(x)\le \int_{G_1}\left(\int_{G_2} |f(x,y)| d\om_{G_2}(y)\right)d\om_{G_1}(x)<\infty.
\end{eqnarray*}
We can now use Fubini's Theorem and \eqref{eq:IT1} to expand \eqref{eq:22} to
\begin{eqnarray*}
f(x,y)&=&\int_{\supp(\nu_2)'}\psi(y)\left(\int_{\supp(\nu_1)}\varphi(x)\mathcal F_{G_1\times G_2}f(\varphi,\psi)d\nu_1(\varphi)\right) d\nu_2(\psi)\\
&=&\int_{\supp(\nu_2)'}\psi(y)p_{G_1}[f,\psi](x)\, d\nu_2(\psi),
\end{eqnarray*}
which is the second equality in \eqref{eq:fxy}.

\medskip
Assume next that (ii) holds, {\em i.e.}, for almost all $\psi\in\supp(\nu_2)$ the function  $x\mapsto \mathcal F_{G_2}f(x,\cdot)(\psi)$ is equal for almost all $x\in G_1$
to a function $p_{G_1}[f,\psi]\in\mathcal P_{K_1}^\sharp(G_1)$. As before the latter is integrable on $G_1$, so by the Integrable Bochner Theorem~\ref{thm:God-int}
$$
\mathcal F_{G_1}(p_{G_1}[f,\psi])(\varphi)=\int_{G_1} \overline{\varphi(x)}p_{G_1}[f,\psi](x)\,d\om_{G_1}(x) \ge 0,\quad \varphi\in \supp(\nu_1),
$$
{\em i.e.},
$$
\int_{G_1'}\overline{\varphi(x)}\left(\int_{G_2}\overline{\psi(y)}f(x,y)\,d\om_{G_2}(y)\right)\,d\om_{G_1}(x)=
\mathcal F_{G_1\times G_2}f(\varphi,\psi)\ge 0
$$
for almost all $\psi\in\supp(\nu_2)$ and all $\varphi\in\supp(\nu_1)$. Since $\mathcal F_{G_1\times G_2}f$ is continuous on $Z_1\times Z_2$, we get $\mathcal F_{G_1\times G_2}f\ge 0$ on $\supp(\nu_1)\times\supp(\nu_2)$, so $f$ is positive definite by Theorem~\ref{thm:God-int}.
\end{proof}

\begin{rem} {\rm Theorem~\ref{thm:Gn} of Gneiting is a special case of Theorem~\ref{thm:CHG*} for the abelian Gelfand pairs $(G_1,K_1)=(\R^k,\{0\})$ and $(G_2,K_2)=(\R^l,\{0\})$ and the equivalence of (i) and (iii). Theorem~\ref{thm:CHG*} also yields the clarifying details missing in the formulation of Gneiting's Theorem.  
}
\end{rem}

\begin{rem} {\rm Let $f$ satisfy the equivalent conditions of Theorem~\ref{thm:CHG*},
{\em i.e.},
$$
f\in\mathcal  P(G_1\times G_2)\cap L^1(\om_{G_1\times G_2})_{K_1\times K_2}^\sharp.
$$
Then $f(\cdot,e_{G_2}) \in \mathcal P(G_1)_{K_1}^\sharp$. It is possible to construct examples so that $f(\cdot,e_{G_2})\notin L^1(\om_{G_1})$, {\em i.e.}, $e_{G_2}\notin G_2'$.}
\end{rem}

In case the first Gelfand pair in Theorem~\ref{thm:CHG*} is compact, we get the following:

\begin{cor}\label{thm:cor1} Let the assumptions in Theorem~\ref{thm:CHG*} hold, and assume that $G_1$ is compact.

Then $G_2'=G_2$ (defined in \eqref{eq:Hnull*2}) and \eqref{eq:FLh} is defined for all $(y,\varphi)\in G_2\times Z_1$ and is continuous from $G_2\times Z_1$ to $\C$. 

The following conditions are equivalent:

(i) $f\in\mathcal P(G_1\times G_2)$.

(ii)  For almost  all $\psi\in\supp(\nu_2)$   
the function $x\mapsto \mathcal F_{G_2} f(x,\cdot)(\psi)$ (defined for $x\in G_1'$) 
is equal almost everywhere in $G_1$ to a $K_1$ bi-invariant continuous positive definite function denoted $p_{G_1}[f,\psi]$ on $G_1$. 

(iii) $p_{G_2}[f,\varphi](y):=\mathcal F_{G_1}f(\cdot,y)(\varphi),\;y\in G_2$ belongs to $\mathcal P_{K_2}^\sharp(G_2)$ for each $\varphi\in Z_1$.

If the equivalent conditions hold, then $p_{G_2}[f,\varphi] \in L^1(\om_{G_2})$ for each $\varphi\in Z_1$ and for $(x,y)\in G_1\times G_2$
\begin{equation}\label{eq:fxycp}
f(x,y)=\sum_{\varphi\in Z_1}\delta(\varphi)\varphi(x)p_{G_2}[f,\varphi](y)=
\int_{\supp(\nu_2)}\psi(y)p_{G_1}[f,\psi](x)\,d\nu_2(\psi),
\end{equation}
where the sum is absolutely and uniformly convergent and the last integrand is defined for $\nu_2$-almost all $\psi\in\supp(\nu_2)$.
\end{cor}	 

\begin{proof} If $G_1$ is compact, then it is known that $Z_1$ is discrete and $\supp(\nu_1)=Z_1$, cf. 
Remark~\ref{thm:reducedsupp}. For $y\in G_2$, $f(\cdot,y)$ is continuous on $G_1$, and hence integrable so $G_2'=G_2$.   Furthermore, it is enough to verify that \eqref{eq:FLh} is continuous for $y\in G_2$ for each fixed $\varphi\in Z_1$. However, if $y_j\to y \in G_2$, then $f(x,y_j)\to f(x,y)$ uniformly for $x\in G_1$ and the result follows.

The first formula in \eqref{eq:fxycp} holds in particular for $x=e_{G_1}, y=e_{G_2}$ and this shows that the series is absolutely and uniformly convergent.
\end{proof}

In case both Gelfand pairs in Theorem~\ref{thm:CHG*} are compact, we get the following:

\begin{cor}\label{thm:2GPcomp} Let $(G_1,K_1)$ and $(G_2, K_2)$ be compact Gelfand pairs with dual spaces $Z_1$ and $Z_2$, and let $f:G_1\times G_2\to \C$ be a continuous function, which is bi-invariant under $K_1\times K_2$. 

The following conditions are equivalent:

(i) $f\in \mathcal P(G_1\times G_2)$.

(ii) $p_{G_1}[f,\psi](x):=\mathcal F_{G_2}f(x,\cdot)(\psi)\in\mathcal P^\sharp_{K_1}(G_1)$ for all $\psi\in Z_2$. 

(iii)  $p_{G_2}[f,\varphi](y):=\mathcal F_{G_1}f(\cdot,y)(\varphi)\in\mathcal P_{K_2}^\sharp(G_2)$ for all $\varphi\in Z_1$. 

If the equivalent conditions hold, then we have 
\begin{equation}\label{eq:fxycp2}
f(x,y)=\sum_{\varphi\in Z_1}\delta(\varphi)\varphi(x)p_{G_2}[f,\varphi](y)=\sum_{\psi\in Z_2}\delta(\psi)\psi(y)p_{G_1}[f,\psi](x),\quad (x,y)\in G_1\times G_2
\end{equation}
and both epansions are absolutely and uniformly convergent on $G_1\times G_2$. 
\end{cor}

\subsection{Products of Gelfand pairs and locally compact groups}

Let $(G,K)$ be a Gelfand pair and let
$L$ be an arbitrary locally compact group. The set of continuous functions $f:G\times L\to \C$ which are bi-invariant with respect to $K$ in the first variable is denoted  $C^\sharp_K(G,L)$.
As explained in Section 2 this space can be identified with the set of continuous functions from $(K\backslash  
G/K)\times L$ to $\C$, where $K\backslash  
G/K$ is the set of double cosets. 
Let
$$
\mathcal P_K^\sharp(G,L):=C_K^\sharp(G,L)\cap \mathcal P(G\times L)
$$
denote the set of continuous positive definite functions on $G\times L$ which are bi-invariant with respect to $K$ in the first variable.

We now extend  Proposition~\ref{thm:CH} by Replacing $\R^d$ by an arbitrary Gelfand pair $(G,K)$ and $\R$ by an arbitrary  locally compact group $L$.

\begin{prop}\label{thm:CHgp} Let $(G,K)$ be a Gelfand pair with dual space $Z$ and  Plancherel measure $\nu$, and let $L$ be a locally compact group.
Let $h:\supp(\nu)\times L\to \C$ be a continuous function satisfying

(i) For each $\varphi\in\supp(\nu)$ we have
$h(\varphi,\cdot)\in\mathcal P(L)$.

(ii) $\int_{Z} h(\varphi,e_{L}) d\nu(\varphi)<\infty.$

\medskip
Then $C:G\times L\to\C$ defined by
$$
C(x,u)=\int_{Z} \varphi(x) h(\varphi,u) d\nu(\varphi)
$$
belongs to $\mathcal P^\sharp_K(G, L)$.
\end{prop}

\begin{proof} Note that 
$(x,u)\mapsto \varphi(x)h(\varphi,u)$ belongs to $\mathcal P^\sharp_K(G, L)$ for each $\varphi\in\supp(\nu)$. This follows from
well-known stability properties of positive definiteness, see e.g. \cite[Proposition 3.2]{B:P:P}.
Furthermore, $|\varphi(x)h(\varphi,u)|\le h(\varphi,e_{L})$, so $C$ is well-defined because of (ii), and also positive definite in the sense of \eqref{eq:pdpointnw}.

We shall finally prove the continuity of $C$ and by a classical property of positive definite functions, it is enough to prove continuity at $(e_{G},e_{L})$. For given $\varepsilon>0$ we choose a compact set $\Gamma\subset \supp(\nu)$ such that
$$
\int_{\supp(\nu)\setminus \Gamma} h(\varphi,e_{L}) d\nu(\varphi)<\varepsilon.
$$  
We next choose a neighbourhood $U\times V$ of $(e_{G},e_{L})$ in $G\times L$ such that for $(x,u,\varphi)\in U\times V\times \Gamma$
$$
|\varphi(x)h(\varphi,u)-h(\varphi,e_{L})|\le \frac{\varepsilon}{\nu(\Gamma)},
$$
which is possible  because $\varphi(x) h(\varphi,u)$ is continuous for $(x,u,\varphi)\in G\times L\times Z$.

For $(x,u)\in U \times V$ we then get
\begin{eqnarray*}
C(x,u)-C(e_{G},e_{L})&=&\int_{\supp(\nu)\setminus\Gamma}  \left(\varphi(x) h(\varphi,u)-h(\varphi,e_{L})\right)d\nu(\varphi)\\
&+&\int_{\Gamma} \left(\varphi(x) h(\varphi,u)-h(\varphi,e_{L})\right)d\nu(\varphi).
\end{eqnarray*}
The first  and the second integral are in absolute value bounded by respectively $2\varepsilon$ and $\varepsilon$, hence
$$
|C(x,u)-C(e_{G},e_{L})|\le 3\varepsilon,
$$
and the continuity follows.
\end{proof}

Suppose now that $(G,K)$ is a compact Gelfand pair and $L$ is an arbitrary locally compact group. For $f\in C_K^\sharp(G,L)$ and $u\in L$ the function $f(\cdot,u)$ belongs to $C_K^\sharp(G)\subset L^2(G)_K^\sharp$ and has an expansion
\begin{equation}\label{eq:orex}
f(x,u)\sim\sum_{\varphi\in Z} \delta(\varphi)\mathcal F f(\cdot,u)(\varphi)\varphi(x),
\end{equation}
which converges in $L^2(G)$ because $(\sqrt{\delta(\varphi)} \varphi)_{\varphi\in Z}$ is an orthonormal basis in $L^2(G)_K^\sharp$, cf. \cite[Theorem 2.6]{B:P:P}. The expansion coefficient functions are
\begin{equation}\label{eq:coeffuB}
B(\varphi)(u):= \delta(\varphi)\mathcal F f(\cdot,u)(\varphi)=\delta(\varphi)\int_G\overline{\varphi(x)} f(x,u)\,d\om_G(x),\quad u\in L.
\end{equation}
Clearly $B(\varphi): L\to\C$ is continuous.

The main result, Theorem 3.3 of \cite{B:P:P}, states the following.

\begin{thm}\label{thm:main}[Berg-Peron-Porcu 2018] Let $(G,K)$ denote a compact Gelfand pair, let $L$ be a locally compact group and let  $f:G\times L\to \C$ be a continuous function. Then
$f$ belongs to $\mathcal P_K^\sharp(G,L)$ if and only if the expansion functions $B(\varphi)$ given by \eqref{eq:coeffuB} satisfy

(i) $B(\varphi)\in \mathcal P(L),\;\varphi \in Z$,

(ii) $\sum_{\varphi\in Z} B(\varphi)(e_L)<\infty$.

If the equivalent conditions hold, then we have
\begin{equation}\label{eq:expand}
f(x,u)=\sum_{\varphi\in Z} B(\varphi)(u)\varphi(x),\quad x\in G,u\in L,
\end{equation}
and the sum is absolutely and uniformly convergent for $(x,u)\in G\times L$.
\end{thm}

When $L$ denotes the group consisting just of the neutral element, Theorem~\ref{thm:main} reduces to the Bochner-Godement Theorem for compact Gelfand pairs.

If we  apply the Bochner-Godement Theorem to the compact Gelfand pair $(G_1\times G_2,K_1\times K_2)$ of Corollary~\ref{thm:2GPcomp}, we can add the following fourth condition equivalent to 
$f\in\mathcal P(G_1\times G_2)$:
 
\begin{equation}\label{eq:2GPcomp1}
(iv)\quad f(x,y)=\sum_{\varphi\in Z_1}\sum_{\psi\in Z_2} B(\varphi,\psi)\varphi(x)\psi(y),\quad (x,y)\in G_1\times G_2.
\end{equation}
The above expansion is uniformly absolutely convergent and the non-negative numbers
$B(\varphi,\psi)$ are given  by
$$
B(\varphi,\psi)=\delta(\varphi)\delta(\psi)\int_{G_1\times G_2}\overline{\varphi(x)\psi(y)}f(x,y) d\om_{G_1\otimes G_2}(x,y),\quad \varphi\in Z_1,\psi\in Z_2.
$$

When $L$ of Theorem~\ref{thm:main} is abelian, $A=L$, we get the following result: The first part is a reformulation of Corollary~\ref{thm:cor1}, if we  consider $A$ as an abelian Gelfand pair $(A,\{0\})$. The second part gives a sharper result under the condition \eqref{eq:g*}.

\begin{thm}\label{thm:PW-GP} Let $(G,K)$ denote a compact Gelfand pair, let $A$ denote an LCA-group and let $f:G\times A\to \C$ be a continuous function bi-invariant with respect to $K$ in the first variable and integrable with respect to $\om_G\otimes\om_A$.

Then the expansion functions $B(\varphi),\varphi\in Z$ given by \eqref{eq:coeffuB} belong to $L^1(\om_A)$.

Define
\begin{equation}\label{eq:h1*}
\mathcal F_A f(x,\cdot)(\gamma):=\int_A \overline{\gamma(u)} f(x,u)\,d \om_A(u),\quad \gamma\in\widehat{A}, x\in G',
\end{equation}
where as before 
$$
G':=\{x\in G \mid f(x,\cdot)\in L^1(\om_A)\}
$$
is a bi-invariant set such that $\om_A(G\setminus G')=0$. 

The following conditions are equivalent:

(i) $f\in\mathcal P_K^\sharp(G,A)$.

(ii) For almost all $\gamma\in \widehat{A}$ the function $\mathcal F_A f(x,\cdot)(\gamma)$ defined for $x\in G'$ is equal almost everywhere to a function in $\mathcal P_K^\sharp(G)$ denoted $p_G[f,\gamma]$.

(iii) $B(\varphi)\in\mathcal P(A)$ for each $\varphi\in Z$.

If the equivalent conditions hold, then
\begin{equation}\label{eq:fxycpA}
f(x,u)=\sum_{\varphi\in Z} B(\varphi)(u)\varphi(x)=
\int_{\widehat{A}}\gamma(u)p_{G}[f,\gamma](x)\,d\om_{\widehat{A}}(\gamma),\quad (x,u)\in G\times A,
\end{equation}
where the sum is absolutely and uniformly convergent, and  the last integrand is defined for almost all $\gamma\in\widehat{A}$.

\medskip
In case there exists a function $\ell \in L^1(\om_A)$ such that
\begin{equation}\label{eq:g*}
|f(x,u)|\le \ell(u),\quad (x,u) \in G\times A,
\end{equation}
then $G'=G$ and $\mathcal F_A f(x,\cdot)(\gamma)$ given by \eqref{eq:h1*} is continuous on $G\times\widehat{A}$.
In this case condition (ii)  simplifies to

(ii') For all $\gamma\in \widehat{A}$ the function $p_G[f,\gamma](x):= \mathcal F_A f(x,\cdot)(\gamma)$ belongs to $\mathcal P_K^\sharp(G)$.
\end{thm}

\begin{proof}
 It is easy to see from  \eqref{eq:coeffuB} that  $B(\varphi)$ is continuous and integrable on $A$ for each $\varphi\in Z$.

Condition  \eqref{eq:g*} clearly  implies that the expression \eqref{eq:h1*}  is well-defined for $(x,\gamma)\in G\times\widehat{A}$.
 To see the continuity we proceed as follows:

Fix $\varepsilon>0$ and choose a compact set $J\subset A$ such that $\int_{A\setminus J} \ell(u)\,d\om_A(u)<\varepsilon$. For $\delta>0$ to be specified later,  we choose a neighbourhood $V$ of $e_G\in G$ such that for $x,y\in G$ satisfying $x^{-1}y\in V$ we have  $|f(x,u)-f(y,u)|<\delta$ when $u\in J$. This is possible because $f$ is uniformly continuous on $G\times J$. We next consider $\gamma,\chi\in\widehat{A}$ with $\sup_{u\in J}|\gamma(u)-\chi(u)|<\delta$. For such $\gamma,\chi$ and $x,y\in G$ with
$x^{-1}y\in V$ we get 
\begin{eqnarray*}
\mathcal F_Af(x,\cdot)(\gamma)-\mathcal F_Af(y,\cdot)(\chi) &=&
\int_A \left(\overline{\gamma(u)}-\overline{\chi(u)}\right)f(x,u)\,d \om_A(u)\\
&+& \int_A \overline{\chi(u)}\left(f(x,u)-f(y,u)\right)\,d \om_A(u),
\end{eqnarray*}
and hence
\begin{eqnarray*}
\lefteqn{\left|\mathcal F_Af(x,\cdot)(\gamma)-\mathcal F_Af(y,\cdot)(\chi)\right|}\\ 
&\le&\int_A \left|\overline{\gamma(u)}-\overline{\chi(u))}\right| \ell(u)\,d \om_A(u)
+ \int_A |(f(x,u)-f(y,u)|\,d \om_A(u)\\
&\le& \int_J \left|\overline{\gamma(u)}-\overline{\chi(u)}\right| \ell(u)\,d \om_A(u)
+\int_J |(f(x,u)-f(y,u)|\,d \om_A(u)+4\varepsilon\\
&\le& \delta\int_J \ell(u)\,d \om_A(u) +\delta \om_A(J)+4\varepsilon
\le \delta\left(\int \ell d\om_A + \om_A(J)\right) +4\varepsilon
  \le 5\varepsilon,
\end{eqnarray*}
if $\delta$ is specified to be less than $\varepsilon(\int \ell(u) \,d \om_A(u) +\om_A(J))^{-1}$. 

(Notice that in case $\widehat{A}$  and $G$ are metrizable, the continuity of $\mathcal F_Af(x,\cdot)(\gamma)$ follows simply from \eqref{eq:g*} and the Theorem of Lebesgue on dominated convergence.)
\end{proof}

\section{The Porcu-White Theorem}

As background material for this section one can consult \cite[Chap. 9]{F}.

We now consider the compact Gelfand pair $(O(d+1),O(d))$, where $O(d)$ is the compact group of orthogonal $d\times d$ matrices. The homogeneous space $O(d+1)/O(d)$ can be identified with the unit sphere $\S^d$ in $\R^{d+1}$ in the following way.

The compact group $G=O(d+1)$ operates in $\R^{d+1}$ and in $\S^d$. 
We use the notation $e_1,\ldots,e_{d+1}$ for the standard basis in $\R^{d+1}$.
The fixed-point group of the matrices $g\in O(d+1)$  satisfying $ge_{1}=e_{1}$,
is of the form
$$
g=\begin{pmatrix}
1 & 0 \\
0 & \tilde g
\end{pmatrix}, 
$$
where $\tilde g\in O(d)$, the zero in the upper right corner represents a zero row vector of length $d$, and the zero in the lower left corner represents a zero column vector of length $d$.

 This shows that the fixed-point group $K$  of $e_{1}$ is isomorphic to $O(d)$ and in the following identified with $O(d)$. The mapping $g\mapsto ge_1$ of
$G=O(d+1)$ onto $\S^d$ is constant on the left cosets $gK$, and hence induces a bijection of $G/K$ onto $\S^d$, and it is a homeomorphism. 

The mapping $g\mapsto ge_1\cdot e_1$ of $G$ onto $[-1,1]$ is constant on the double cosets, and if $ge_1\cdot e_1=he_1\cdot e_1$ for $g,h\in G$, then they belong to the same double coset. Therefore the space of double cosets $K\backslash G/K$ is homeomorphic to $[-1,1]$. This shows that complex functions on $G$ which are bi-invariant with respect to $K$, can be identified with functions  $f:[-1,1]\to\C$. In fact, for such a function, $g\mapsto f(ge_{1}\cdot e_{1})$ is a bi-invariant function on $G$, and all bi-invariant functions on $G$ have this form. The bi-invariant functions depend only on the upper left corner $g_{1,1}$ of $g\in O(d+1)$.

The surface measure of $\S^d$ is denoted $\omega_d$, and it is of total mass
\begin{equation}\label{eq:mass} 
\sigma_d:=\omega_d(\S^d)=\frac{2\pi^{(d+1)/2}}{\Gamma((d+1)/2)}.
\end{equation}
Furthermore, we have
$$
\int_{-1}^1 (1-x^2)^{d/2-1}\,dx=\sigma_d/\sigma_{d-1}.
$$
In the following we consider the probability measure $\tau_d$ on $[-1,1]$ given by the weight function
\begin{equation}\label{eq:weight-d}
(\sigma_{d-1}/\sigma_d) (1-x^2)^{d/2-1}, \quad -1<x<1.
\end{equation}

The image measure of normalized Haar measure $\omega_G$ on $G=O(d+1)$ under the mapping $g\mapsto ge_1$ of $G$ onto $\S^d$ is the normalized surface measure $\omega_d/\sigma_d$. The image measure of $\omega_G$ under the mapping $g\mapsto ge_1\cdot e_1$ of $G$ onto $[-1,1]$ is the probability measure $\tau_d$.

The positive definite spherical functions are
the normalized Gegenbauer polynomial  $c_n(d,x), n=0,1,\ldots$, given by

 \begin{equation}\label{eq:nor}
c_n(d,x)=C_n^{((d-1)/2)}(x)/C_n^{((d-1)/2)}(1)=\frac{n!}{(d-1)_n}C_n^{((d-1)/2)}(x),
\end{equation}     
where $C_n^{(\lambda)}$ are the classical Gegenbauer polynomials in the notation of \cite{A:A:R}. Furthermore,  $(a)_n:=a(a+1)\cdots (a+n-1)$ is the Pochhammer symbol.

 A spherical harmonic of degree $n$ for $\mathbb S^d$  is  the restriction to $\mathbb S^d$ of a real-valued harmonic homogeneous polynomial in $\mathbb R^{d+1}$ of degree $n$. Together with the zero function, the spherical harmonics of degree $n$  form a finite dimensional vector space denoted $\mathcal H_n(d)$. It is a subspace of 
the space  ${\mathcal C}(\mathbb S^d)$ of continuous functions on $\mathbb S^d$ and can be identified with the space $H_\varphi$ corresponding to the spherical function $\varphi=c_n(d,\cdot)$.  We have
\begin{equation}\label{eq:dim}
\delta(c_n(d,\cdot))=N_n(d):=\dim \mathcal H_n(d)=\frac{(d)_{n-1}}{n!}(2n+d-1),\;n\ge 1,\quad N_0(d)=1.
\end{equation} 

Let now $A$ be an LCA-group and let $f:[-1,1]\times A\to\C$ be a continuous function.
For the spherical function $c_n(d,\cdot)$ the expansion function $B(c_n(d,\cdot))$ from \eqref{eq:coeffuB} with $L=A$ is called the $d$-Schoenberg function of $f$ in \cite{B:P}, where it is denoted $b_{n,d}$. Formula \eqref{eq:coeffuB} can be reduced to
\begin{equation}\label{eq:dSch}
b_{n,d}(u)=N_n(d)\int_{-1}^1 f(x,u)c_n(d,x) d\tau_d(x),\quad u\in A.
\end{equation}

Note that using the terminology of \cite{B:P} we have 
$$
\mathcal P_{O(d)}^\sharp(O(d+1),A)=\mathcal P(\S^d,A),
$$
and functions $f$ from these spaces can be considered as functions $f:O(d+1)\times A\to \C$  which are bi-invariant with respect to $O(d)$ in the first variable or as functions $f:[-1,1]\times A\to \C$. Similarly 
$$
\mathcal P_{O(d)}^\sharp(O(d+1))=\mathcal P(\S^d),
$$
and functions $f$ from these spaces can be considered as functions $f:O(d+1)\to \C$  which are bi-invariant with respect to $O(d)$  or as functions $f:[-1,1]\to \C$.

Specializing Theorem~\ref{thm:PW-GP} to the compact Gelfand pair $(O(d+1),O(d))$ we get:

\begin{thm}\label{thm:PW} Let $A$ be an LCA-group and let $f:[-1,1]\times A\to\C$ be a continuous function, assumed integrable with respect to the product measure $\tau_d\otimes \om_A$.  The $d$-Schoenberg functions $b_{n,d}$ given by \eqref{eq:dSch} belong to $L^1(\om_A)$. Let $N\subset[-1,1]$ denote the $\tau_d$-null set such that $f(x,\cdot)\in L^1(\om_A)$ for $x\in [-1,1]\setminus N$ and define
\begin{equation}\label{eq:h1}
\mathcal F_A f(x,\cdot)(\gamma):=\int_A \overline{\gamma(u)} f(x,u)\,d \om_A(u),\quad \gamma\in\widehat{A}, x \in [-1,1]\setminus N.
\end{equation}

The following conditions are equivalent:

(i) $f\in\mathcal P(\S^d,A)$.

(ii) For almost all $\gamma\in \widehat{A}$ the function $\mathcal F_A f(x,\cdot)(\gamma)$ defined  for $x\in[-1,1]\setminus N$ is equal $\tau_d$-almost everywhere to a function $p[f,\gamma]\in \mathcal P(\S^d)$.

(iii) $b_{n,d}\in\mathcal P(A)$ for each $n\ge 0$.

If the equivalent conditions hold, then for $(x,u)\in [-1,1]\times A$
\begin{equation}\label{eq:fxu}
f(x,u)=\sum_{n=0}^\infty b_{n,d}(u)c_n(d,x)= \int_{\widehat{A}}\gamma(u) p[f,\gamma](x)\,d\om_{\widehat{A}}(\gamma),
\end{equation}
where the sum is absolutely and uniformly convergent, and the last integrand is defined for almost all $\gamma \in\widehat{A}$.

\medskip
In case there exists a function
$\ell\in L^1(\om_A)$ such that  
\begin{equation}\label{eq:g}
|f(x,u)|\le \ell(u),\quad x\in[-1,1], u\in A,
\end{equation}
then $N=\emptyset$ and $\mathcal F_A f(x,\cdot)(\gamma)$ given by \eqref{eq:h1} 
is continuous on $[-1,1]\times\widehat{A}$. The $\om_{\widehat{A}}$ null-set in (ii) can be chosen as the empty set.
\end{thm}

\begin{rem} {\rm In \cite[Theorem 1]{P:W} Porcu and White proved Theorem~\ref{thm:PW} in the special case of $A=\R$ and with a special function $\ell$ satisfying \eqref{eq:g}.

They consider the function
\begin{equation}\label{eq:bd}
\ell(u)=B_d(u):=\sum_{n=0}^\infty |b_{n,d}(u)|,\quad u\in A,
\end{equation}
and they assume $B_d\in L^1(\om_A)$.

Let us prove that 
\begin{equation}\label{eq:in1}
|f(x,u)|\le B_d(u),\quad (x,u)\in [-1,1]\times A.
\end{equation}

In fact for fixed $u\in A$ the continuous function $f(\cdot,u)$ has the expansion
$$
\sum_{n=0}^\infty b_{n,d}(u) c_n(d,x)
$$ 
in terms of the orthogonal polynomials $c_n(d,x)$, and the series converges in $L^2(\tau_d)$. By a classical theorem there exists a sequence $n_j=n_j(u), j=1,2,\ldots$ of natural numbers tending to infinity
such that
$$
\lim_{j\to\infty}\sum_{n=0}^{n_j} b_{n,d}(u) c_n(d,x)=f(x,u)
$$
for almost all $x\in[-1,1]$, and hence for those $x$
\begin{eqnarray*}
|f(x,u)|&=&\lim_{j\to\infty}\left|\sum_{n=0}^{n_j} b_{n,d}(u) c_n(d,x)\right|\\
         &\le& \lim_{j\to\infty}\sum_{n=0}^{n_j} |b_{n,d}(u)|=B_d(u).
\end{eqnarray*}
By continuity in $x$ we then get \eqref{eq:in1}. 
}
\end{rem}

In \cite{P:C:G:W:A} the authors consider positive definite functions on products of  spheres $\S^{d_1}$ and $\S^{d_2}$. Their Theorem B.1 is a special case of Corollary~\ref{thm:2GPcomp} applied to the product of the two compact Gelfand pairs
$(O(d_i+1),O(d_i)), i=1,2$.

\section{Revisiting a Theorem of Gneiting}  

In Schoenberg's fundamental paper \cite{Sc} there is a characterization of the class
$$
\mathcal  P(\S^\infty):=\cap_{d=1}^\infty \mathcal P(\S^d)
$$
as the power series
\begin{equation}\label{eq:Schinf}
f(x)=\sum_{n=0}^\infty b_n x^n,\quad x\in [-1,1]
\end{equation}
where $b_n\ge 0$ and $f(1)=\sum b_n<\infty$.

Gneiting \cite{Gn1} introduced the class $\Psi_\infty$ of continuous functions $\psi:[0,\pi]\to\R$ of the form $\psi(\theta)=f(\cos(\theta))$ with $f\in\mathcal P(\S^\infty)$ and $f(1)=1$. 

Theorem 7 in \cite{Gn1} states the following:
\begin{thm}[Gneiting]\label{thm:Gne} Let $\rho:[0,\infty)\to\R$ be a completely monotonic function with $\rho(0)=1$. Then the restriction of $\rho$ to $[0,\pi]$ belongs to $\Psi_\infty$.
\end{thm}

\begin{rem} {\rm The statement of Gneiting's Theorem is actually that the restriction to $[0,\pi]$ of a non-constant completely monotonic function belongs to $\Psi_\infty^+$, where the plus sign refers to the positive definiteness being strict. We shall not consider this aspect and therefore the assumption of being non-constant is not important. The purpose of this section is  to give a new proof of Theorem~\ref{thm:Gne} and the new proof has the advantage that we obtain information about the power series coefficients $b_n$ from \eqref{eq:Schinf}.
}
\end{rem}

\begin{rem} {\rm Schoenberg \cite{Sc} proved that
$$
\lim_{d\to\infty} c_n(d,x)=x^n,\quad -1<x<1,
$$
where $c_n(d,x)$ are the spherical functions for $(O(d+1),O(d))$ from Eq. \eqref{eq:nor}. This has been generalized to inductive limits of suitable sequences $(G_n,K_n)$ of Gelfand pairs, see
\cite{D:O:W} and references therein. 
}
\end{rem}

{\it Proof of Theorem~\ref{thm:Gne}:}
 
By a theorem of Bernstein, cf. \cite[p. 160]{Wi}, the functions of Theorem~\ref{thm:Gne} have the form
\begin{equation}\label{eq:Gn2}
\rho(\theta)=\int_0^\infty \exp(-a\theta)\,d\mu(a),
\end{equation}
where $\mu$ is a probability measure on $[0,\infty)$. Because of stability properties of the set $\Psi_\infty$, it is enough to prove that $\exp(-a\theta)\in\Psi_\infty$ for each $a>0$.

If we let $\Arccos:[-1,1]\to [0,\pi]$ denote the inverse of $\cos:[0,\pi]\to [-1,1]$, we have to prove that $\exp(-a\Arccos x)$ has non-negative power series coefficients.

We first notice that
\begin{equation}\label{eq:arc}
\Arccos(x)=\pi/2-\sum_{n=0}^\infty \frac{(1/2)_n}{n!}\frac{x^{2n+1}}{2n+1},\quad |x|\le 1.
\end{equation}
 To see \eqref{eq:arc}, we use that the derivative of $\Arccos(x)$ is $-(1-x^2)^{-1/2}$ and
$$
(1-x^2)^{-1/2}=\sum_{n=0}^\infty\binom{-1/2}{n} (-x^2)^n,\quad |x|<1.
$$

For any $a>0$ we then have
\begin{eqnarray*}
\exp(-a\Arccos(x)) &=& e^{-a\pi/2}\exp\left(a\sum_{n=0}^\infty \frac{(1/2)_n}{n!}\frac{x^{2n+1}}{2n+1}\right)\\
&=&
e^{-a\pi/2}\sum_{n=0}^\infty \frac{r_n(a)}{n!}x^n,\quad |x|\le 1,
\end{eqnarray*}
where $r_n(a)>0$ for all $n$. $\square$

\medskip
 Of course the expressions for $r_n(a)$ are complicated, but 
 using the exponential Bell partition  polynomials $B_n$ we can find expressions for the coefficients $r_n(a)$. From \cite[Section 11.2]{Ch} we have
$$
\exp\left(\sum_{k=1}^\infty \frac{a_k}{k!}x^k\right)=\sum_{n=0}^\infty \frac{B_n(a_1,\ldots,a_n)}{n!}x^n.
$$
 It is known that 
$$
B_0=1,\quad B_1(a_1)=a_1,\quad B_2(a_1,a_2)=a_1^2+a_2,
$$
and in general we have the recursion formula
$$
B_{n+1}(a_1,\ldots,a_{n+1})=\sum_{k=0}^n\binom{n}{k}B_{n-k}(a_1,\ldots,a_{n-k})a_{k+1}.
$$
We now use 
$$
a_{2n+1}=a\frac{(1/2)_n(2n)!}{n!}=a((2n-1)!!)^2,\quad a_{2n}=0,
$$
($(2k-1)!!:=1\cdot 3\cdot 5\cdots (2k-1)$) and define $r_n(a):=B_n(a_1,\ldots,a_n)$. We see that the recursion becomes
\begin{equation}
r_{n+1}(a)=a\sum_{k=0}^{[n/2]}\binom{n}{2k} r_{n-2k}(a) ((2k-1)!!)^2,
\end{equation}
so $r_n(a)$ is a monic polynomial in $a$ of degree $n$ with non-negative coefficients. The first polynomials are given by
$$
r_0(a)=1,\quad r_1(a)=a,\quad r_2(a)=a^2, \quad r_3(a)=a^3+a,\quad r_4(a)=a^4+4a^2
$$ 
$$
r_5(a)=a^5+10a^3+9a,\quad r_6(a)=a^6 + 20 a^4+64 a^2,\quad r_7(a)=a^7+35a^5+259a^3+225a.
$$
For the completely monotonic normalized function $\rho$ given by \eqref{eq:Gn2} we find
$$
\rho(\theta)=\sum_{n=0}^\infty \frac{c_n}{n!}\cos^n(\theta),\quad c_n=\int_0^\infty e^{-a\pi/2} r_n(a)\,d\mu(a). 
$$

\section{Appendix} We shall give an example showing that the functions  $C(\mathbf{h};u)$ constructed in Proposition~\ref{thm:CH} need not be integrable.

 In the following $\mathcal F f$ denotes the Fourier transform of a function $f:\R\to \C$ given by
$$
\mathcal F f(t)=\Int e^{-itx}f(x)\,dx,\quad t\in\R.
$$
Let $C_0(]-1,1[)$ denote the set of continuous functions $f:]-1,1[\to\C$ vanishing at "infinity", {\em i.e.}, 
$$
\lim_{x\to -1}f(x)=\lim_{x\to 1} f(x)=0.
$$
It is a Banach space under the uniform norm
$$
||f||=\sup\{|f(x)|\mid -1<x<1\}.
$$

 We proceed in a number of steps.
 
{\bf 1:} {\it There exists $f\in C_0(]-1,1[)$ such that $\mathcal F f\notin L^1(\R)$}. 

This is a classical application of the Banach-Steinhaus Theorem to the continuous linear functionals on the Banach space $C_0(]-1,1[)$:
$$
L_n(f)=\int_{-n\pi}^{n\pi} \mathcal F f(t)dt=\int_{-1}^1f(x)\frac{2\sin(n\pi x)}{x}dx,\quad f\in C_0(]-1,1[).
$$
In fact, assuming that $L_n, n\ge 0$ is pointwise bounded, we get that $||L_n||$ is bounded, which  is a contradiction because
$$
||L_n||=\int_{-1}^1\left|\frac{2\sin(n\pi x)}{x}\right|dx=4\int_0^{n\pi}\frac{|\sin(u)|}{u}du
$$
tends to infinity for $n\to\infty$. This shows the existence of an $f\in C_0(]-1,1[)$ such that $(L_n(f))$ is an unbounded sequence, and in particular $\mathcal F f\notin L^1(\R)$.

\medskip
{\bf 2:}  {\it There exists $f\in C_0^+(]-1,1[)$ with $\max f=1$ such that $\mathcal F f\notin L^1(\R)$}. 
 
This is an easy consequence of {\bf 1}.

\medskip
{\bf 3:} {\it There exists $f\in C_c^+(\R)$ with $\max f=1$ and $f(x)>0$ for $x\in[0,1]$ such that $\mathcal F f\notin L^1(\R)$}.  

For $f$ as in {\bf 2} let $x_0\in \R$ satisfy $f(x_0)=1$. Then  there exists $\delta>0$ such that $f(x)>0$ for $x\in[x_0,x_0+\delta]$, and
$x\mapsto f(\delta x+x_0)$ satisfies {\bf 3}.

\medskip
{\bf 4:} {\it There exists $f\in C(\R)\cap L^1(\R)$ such that $0<f(x)<1$ for all $x\in\R$ and $\mathcal F f\notin L^1(\R)$}. 

Let $h$ have the properties of {\bf 3} and let $a_n>0, n\in\Z$ be such that $a_0=1/2$, $\sum_{n\in\Z,n\neq 0} a_n=1/4$.
Then
$$
f(x)=\sum_{n\in\Z} a_n h(x-n),\quad x\in\R 
$$
is continuous and has the properties $0<f(x)\le 3/4$ for $x\in\R$ and $\int f(x)dx=3/4\int h(x)dx<\infty$. Furthermore,
$$
\mathcal F f(t)=\mathcal F h(t)\sum_{n\in\Z}a_n e^{-itn},
$$
hence
$$
|\mathcal F f(t)|=|\mathcal F h(t)|\left|\sum_{n\in\Z}a_n e^{-itn}\right| \ge |\mathcal F h(t)|\left(a_0-\sum_{n\in\Z,n\neq 0}a_n\right)=1/4|\mathcal F h(t)|,
$$
showing that $\mathcal F f\notin L^1(\R)$.

{\bf Conclusion:} Define 
$$
g(\om;\tau)=\frac{1}{\sqrt{2\pi}}f(\om)e^{-\tau^2/2},\quad (\om,\tau)\in\R\times\R,
$$
where $f$ satisfies {\bf 4}. Then $g$ is strictly positive, continuous and integrable and
$$
h(\om;u):=\Int g(\om;\tau)e^{iu\tau}d\tau=f(\om)e^{-u^2/2}
$$
satisfies (C1') and (C2') of Proposition~\ref{thm:CH} with $d=1$ and
$$
C(h;u)=\Int e^{ih\om}h(\om;u)d\om=e^{-u^2/2}\mathcal F f(-h)\notin L^1(\R^2).
$$

\noindent{\bf Acknowledgment} The author wants to thank Emilio Porcu, Zolt{\'a}n Sasv{\'a}ri and Ryszard Szwarc for valuable advice during the preparation of this paper. 
The author also wants to thank two independent referees for valuable comments.

\noindent
Christian Berg\\
Department of Mathematical Sciences, University of Copenhagen\\
Universitetsparken 5, DK-2100, Denmark\\
e-mail: {\tt{berg@math.ku.dk}}

\end{document}